\documentclass[12pt]{amsart}
\usepackage{latexsym,enumerate}
\usepackage{amssymb, xcolor,mathrsfs}
\usepackage{amsmath,amsthm,amsfonts,amssymb,latexsym, mathabx}
\usepackage[pagebackref,colorlinks]{hyperref}
\usepackage[utf8]{inputenc}
\usepackage{arydshln}

\newtheorem{theorem}{Theorem}[section]

\newtheorem{proposition}[theorem]{Proposition}

\newtheorem{conjecture}[theorem]{Conjecture}

\newtheorem{question}[theorem]{Question}

\theoremstyle{definition}

\numberwithin{equation}{section}

\newcommand{\Aut}{{\mathrm {Aut}}}

\newcommand{\Irr}{{\mathrm {Irr}}}

\newcommand{\Syl}{{\mathrm {Syl}}}

\newcommand{\Gal}{{\rm Gal}}

\newcommand{\Char}{{\mathrm {Char}}}

\newcommand{\QQ}{{\mathbb Q}}
\newcommand{\ZZ}{{\mathbb Z}}

\newcommand{\EE}{{\mathbb E}}

\newcommand{\FF}{{\mathbb F}}

\newcommand{\height}{\mathbf{ht}}
\newcommand{\lev}{\mathrm{\mathbf{lev}}}

\newcommand{\bC}{{\mathbf{C}}}

\newcommand{\bN}{{\mathbf{N}}}

\newcommand{\Al}{\textup{\textsf{A}}}

\makeatletter \@namedef{subjclassname@2020}{\textup{2020}
Mathematics Subject Classification} \makeatother

\def\nor{\trianglelefteq\,}

\begin{document}

\title[Problems on character conductors]
{Problems on the conductor of finite group characters}

\author[Nguyen N. Hung]{Nguyen N. Hung}
\address{Department of Mathematics, The University of Akron, Akron,
OH 44325, USA}
\email{hungnguyen@uakron.edu}

\thanks{The author gratefully acknowledges support from the AMS--Simons Research
Enhancement Grant (AWD-000167 AMS). He also thanks Gunter Malle,
Gabriel Navarro, Mandi Schaeffer Fry, and Pham Huu Tiep for their
careful reading and many helpful comments on an earlier version of
the manuscript. Finally, the author thanks the anonymous referee for
several constructive comments and suggestions.}



\thanks{}

\begin{abstract} This paper reviews recent results and open problems on the conductor of finite group characters,
highlighting their connections to one another and to broader topics
in the representation theory of finite groups.
\end{abstract}

\maketitle

\tableofcontents


\section{Introduction}

Studying character values is a fundamental yet challenging problem
in the representation theory of finite groups. For example, it has
long been known that a finite group has a nontrivial real-valued
irreducible character if and only if its order is even. However, the
analogous statement with ``rational" replacing ``real" was only
proved by using the classification of finite simple groups
\cite{Navarro-Tiep08}. The well-known \emph{rational groups} --
those whose characters are all rational-valued -- have been studied
for decades, but their complete structure remains far from fully
understood, see \cite{Thompson08} and the references therein.

In recent years, the study of character values has regained interest
due to significant progress on several fundamental problems in the
field. These include Feit's conjecture \cite{Feit80}, the
McKay-Navarro conjecture and its blockwise refinement
\cite{Navarro04}, and the determination of the field of values of
irreducible characters of degree not divisible by a prime $p$
\cite{Navarro-Tiep21}, among others.

This paper discusses recent results and open problems related to the
\emph{conductor} of complex irreducible characters of finite groups.
Given a character $\chi$ of a finite group $G$, let $\QQ(\chi)$
denote its field of values -- the smallest field extension of $\QQ$
containing all the values of $\chi$. Since each value $\chi(g)$, for
$g\in G$, is a sum of roots of unity, the field of values of a
character is always an abelian number field. The conductor of
$\chi$, denoted $c(\chi)$, is defined as the smallest positive
integer $n$ such that $\QQ(\chi)$ is contained in the $n$-th
cyclotomic field $\QQ(e^{2\pi i/n})$. As an important invariant, the
conductor provides key insights into the rationality/irrationality,
as well as $p$-rationality, of character values.

The role of the character conductor can be explicit, as in Feit's
conjecture mentioned above, but it is often more subtle. For
instance, in their seminal work \cite{Navarro-Tiep21}, G.~Navarro
and P.\,H.~Tiep proved that an abelian number field $\FF$ can arise
as the field of values of an odd-degree irreducible character if and
only if $\FF\supseteq \QQ(e^{2\pi i/{c(\FF)_2}})$, where $c(\FF)_2$
is the $2$-part of the conductor $c(\FF)$ of $\FF$. As we will see,
the conductor also plays a central role in classifying possible
fields of values of $p'$-degree characters, height zero characters,
$p$-Sylow restrictions, and characters of $p'$-order groups, as well
as in determining the $p$-rationality of a character via local
subgroups.

We have two main goals. The first is to collect and review a number
of results and open problems on character conductors, which are
scattered across the literature. The second is to highlight their
connections to one another and to broader topics in the field. We
hope that this review will stimulate further interest in these
problems and, more generally, in the study of character conductors.


\section{Feit's conjecture}\label{sec:conductor}

Throughout, let $p$ be a prime and let $G$ be a finite group. As
usual, we write $\Irr(G)$ for the set of all (complex) irreducible
characters of $G$, and $\Irr_{p'}(G)$ for the subset consisting of
those characters whose degrees are not divisible by $p$. We use
$\Char(G)$ for the set of all characters of $G$. For $n\in\ZZ^+$,
let $\QQ_n$ denote the $n$th cyclotomic field $\QQ(e^{2\pi i/n})$,
and let $\QQ^{ab}$ be the extension of $\QQ$ obtained by adjoining
all roots of unity. If $n\in\ZZ^+$, we denote by $n_p$ the largest
power of $p$ dividing $n$ and set $n_{p'}:=n/n_p$. The $p$-adic
valuation is $\nu_p(n):=\log_p(n_p)$.

Recall that if $\mathbb{F}$ is a finite abelian extension of
$\mathbb{Q}$, its conductor, denoted $c(\mathbb{F})$, is the
smallest positive integer $n$ such that $\mathbb{F} \subseteq
\mathbb{Q}_n$. The conductor $c(\chi)$ of a character $\chi$ is
therefore the conductor of its field of values $\mathbb{Q}(\chi)$;
that is, $c(\chi) := c(\mathbb{Q}(\chi))$. Furthermore, if $z$ is a
sum of roots of unity, we let $c(z)$ denote the conductor of the
smallest field containing $z$.

A basic question regarding the character conductor is: what specific
group-theoretic properties determine the conductors of irreducible
characters of a given group? This problem is quite challenging, even
when restricted to certain families of groups. Among groups that are
structurally close to simple, some, like symmetric groups, are
rational, while classical groups tend to be highly irrational.

Let $\chi\in\Irr(G)$. It is easy to see that, for every $g\in G$,
$\chi(g)$ is a sum of $|g|$-th roots of unity, and it follows that
$\QQ(\chi)\subseteq \QQ_{\exp(G)}$, where $\exp(G)$ denotes the
exponent of $G$. Therefore, $c(\chi)$ is always a divisor of
$\exp(G)$.

Earlier work of H.~Blichfeldt \cite{Bl04}, W.~Burnside \cite{Bu11},
and R.~Brauer \cite{Brauer64} suggests that the conductors of
irreducible characters of a finite group may in fact be tightly
controlled by the orders of group elements. For instance, Brauer
showed that if $m$ is the product of those prime-power factors
$p_i^{a_i}$ of $c(\chi)$ with $a_i>1$, then $G$ has an element of
order $m$ (see \cite[Theorem~3.10]{Navarro18}). These results
motivated a remarkable conjecture of W.~Feit \cite{Feit80}.

\begin{conjecture}[Feit]\label{conj:Feit}
Let $\chi$ be an irreducible character of a finite group $G$. Then
$G$ contains an element of order $c(\chi)$.
\end{conjecture}

Feit's conjecture was proved for solvable groups independently by
G.~Amit and D.~Chillag \cite{Amit-Chillag86} and P.\,A.~Ferguson and
A.~Turull \cite{Ferguson-Turull86}. In fact, a slightly stronger
result was achieved in \cite{Ferguson-Turull86}: if $\chi\in\Irr(G)$
and $G$ is $\pi$-solvable, where $\pi$ is the set of prime divisors
of $\chi(1)$, then $G$ contains an element of order $c(\chi)$.
Moreover, the conjecture holds when the non-abelian composition
factors of $G$ are minimal simple groups \cite{Ferguson-Turull87},
or when the degree of $\chi$ is prime \cite{HTZ}.

In \cite{V16}, C.~Vallejo proved a local strengthening of Feit's
conjecture for odd-degree characters of solvable groups: let $p$ be
a prime, $G$ a solvable group, and $\chi\in\Irr(G)$ with $\chi(1)$
odd and coprime to $p$, then for $P \in \Syl_p(G)$, the quotient
$\bN_G(P)/P'$ contains an element of order $c(\chi)$. This
strengthening, however, fails in general for characters of even
degree or for non-solvable groups.

There are also earlier results on the character conductor that are
not directly related to Feit's conjecture. Recall that $c(\chi)$
always divides the exponent of the group. Another noteworthy
divisibility property is that $c(\chi)$ divides $|G|/\chi(1)$ -- a
conjecture originally due to J.\,H.~Conway and later proved by Feit
\cite{Feit82}. Vallejo \cite{V23} introduced the notion of
\emph{relative conductor} in the setting where the character in
question restricts homogeneously to a normal subgroup and
established a divisibility relation between this relative conductor
and the exponent of the corresponding quotient group. In
\cite{Moreto}, A.~Moret\'{o} proved that the order of a finite group
is bounded from above in terms of the maximal multiplicity of the
conductors of its irreducible characters.

Recently, R.~Boltje, A.\,S.~Kleshchev, G.~Navarro, and P.\,H.~Tiep
have made a breakthrough on Feit's conjecture by reducing it to a
question about finite simple groups. Their key insight was to work
with a stronger formulation involving characters of subgroups,
rather than focusing on element orders. For the next conjecture, we
recall that an irreducible character $\chi$ of $G$ is termed
\emph{primitive} if it is not induced from a character of a proper
subgroup of $G$.

\begin{conjecture}[Boltje-Kleshchev-Navarro-Tiep]\label{conj:BKNT24}
Let $G$ be a finite group and $\chi\in\Irr(G)$. Suppose that $\chi$
is nonlinear and primitive. Then there exist a self-normalizing
proper subgroup $H< G$ and $\psi\in\Irr(H)$ such that $\QQ(\chi)=
\QQ(\psi)$.
\end{conjecture}

Conjecture~\ref{conj:BKNT24} can be reduced to the so-called
\emph{inductive Feit condition} which must be verified for all
finite simple groups. This reduction builds on earlier successful
reductions for local-global conjectures (some of these are discussed
in Section~\ref{sec:p-rationality}). As summarized in
\cite[\S1]{BKNT24}, one has to show that, for the universal covering
group $X$ of each finite simple group and an irreducible character
$\chi$ of $X$, there exists a self-normalizing proper subgroup $U$
of $X$ and an irreducible character $\mu$ of $U$ such that $\chi$
and $\mu$ have the same cohomology class and exhibit the same
behavior under the action of $\Aut(X)_U \times \Gal(\QQ^{ab}/\QQ)$,
where $\Aut(X)_U$ denotes the group of automorphisms of $X$ that
stabilize $U$. We refer the interested reader to \cite[\S3]{BKNT24}
for the precise definition.

Theorem~4.7 of \cite{BKNT24} shows that if every nonabelian
composition factor of $G$ satisfies the inductive Feit condition,
then Conjecture~\ref{conj:BKNT24} holds for $G$. Consequently, one
has the following.

\begin{theorem}[\cite{BKNT24}, Theorem 4.8]
Let $G$ be a finite group. Assume that the simple groups involved in
$G$ satisfy the inductive Feit condition. Then Feit's
Conjecture~\ref{conj:Feit} holds for every irreducible character of
$G$.
\end{theorem}

\noindent (Recall that a simple group \(S\) is said to be
\emph{involved} in a finite group \(G\) if \(S\) occurs as a
subquotient of \(G\); that is, if there exist \(H \le G\) and \(N
\trianglelefteq H\) such that $H/N \cong S$.) The inductive Feit
condition has been verified for the alternating groups, the sporadic
groups, and several families of low-rank groups of Lie type (see
\cite[\S5 \& \S8]{BKNT24} and \cite{Tapp25}). As a consequence,
Feit's conjecture is now known for all finite groups with abelian
Sylow $2$-subgroups.

Since \(c\bigl(\chi(g)\bigr)\mid |g|\), the conductor
\(c\bigl(\chi(g)\bigr)\) is equal to the order of a suitable power
of \(g\). This leads naturally to the following strengthening of
Feit's conjecture: does there exist \(g\in G\) such that $
c(\chi)=c\bigl(\chi(g)\bigr)? $ Equivalently, does there exist
\(g\in G\) such that $ c\bigl(\chi(g)\bigr) $ is divisible by
\(c\bigl(\chi(x)\bigr)\) for every \(x\in G\)? As noted in
\cite[\S9]{BKNT24}, this stronger statement is false in general.
However, known counterexamples seem to be rare and, so far, none are
known for quasi-primitive characters, including all irreducible
characters of quasisimple groups. Here $\chi\in\Irr(G)$ is
\emph{quasi-primitive} if for every normal subgroup $N\nor G$, the
restriction $\chi_N$ is a multiple of an irreducible character. In
recent work \cite{HH26}, Herbig and the author verified this
property for several low-rank families of quasisimple groups of Lie
type.

We end this section with a generalization -- perhaps a more natural
version -- of Feit's conjecture.

\begin{conjecture}[Boltje-Kleshchev-Navarro-Tiep]\label{conj:aaa}
Suppose that $G$ is a finite group and $\chi\in \Irr(G)$. Then there
is a subgroup $H$ of $G$ and a linear constituent $\lambda\in
\Irr(H)$ of the restriction $\chi_H$ such that the conductor of
$\chi$ is the order of $\lambda$.
\end{conjecture}

Note that the order $o(\lambda)$ of such a linear character
$\lambda$ divides $|H/H'|$. Hence $H/H'$ contains an element of
order $o(\lambda)$, which in turn implies that $H$ contains an
element of order $o(\lambda)$. Therefore, Conjecture~\ref{conj:aaa}
indeed implies Conjecture~\ref{conj:Feit}.

Conjecture~\ref{conj:aaa} is known in the case of solvable groups,
by \cite[Theorem~9.4]{BKNT24}, but open in general, even for simple
groups.


\section{The cyclotomic deficiency}\label{sec:cyclotomic-deficiency}

The Galois group $\mathcal{G}:=\Gal(\QQ^{ab}/\QQ)$ of $\QQ^{ab}$,
the smallest extension of $\QQ$ containing all roots of unity,
naturally acts on $\Irr(G)$ by $\chi^\sigma(g):=\chi(g)^\sigma$ for
every $\sigma\in \mathcal{G}$, $\chi\in\Irr(G)$, and $g\in G$.
Studying the sizes of $\mathcal{G}$-orbits is a fundamental problem
in group representation theory. The orbit size of $\chi$ is nothing
more than the order of the Galois group $\Gal(\QQ(\chi)/\QQ)$. It is
linked to the conductor of $\chi$ via
\[
|\Gal(\QQ(\chi)/\QQ)|=\frac{\varphi(c(\chi))}{[\QQ_{c(\chi)}:\QQ(\chi)]},
\]
where $\varphi$ is Euler's totient function (since
$[\QQ_{c(\chi)}:\QQ]=\varphi(c(\chi))$). The degree of the field
extension $\QQ_{c(\chi)}/\QQ(\chi)$ measures how far $\QQ(\chi)$ is
from its cyclotomic closure and often referred to as the
\emph{cyclotomic deficiency} of $\chi$.

In 1988, G.\,M.~Cram \cite{Cram88} proved, using
Gajendragadkar-Isaacs's theory of $p$-special characters of
$p$-solvable groups (see \cite[Chapter~2]{I18}), that if
$\chi\in\Irr(G)$ where $G$ is \emph{solvable}, then
$$[\QQ_{c(\chi)}:\QQ(\chi)] \text{ divides } \chi(1).$$ This divisibility,
unfortunately, no longer holds in arbitrary non-solvable groups.
(For example, the irreducible characters of degree 3 of $\Al_5$ have
field of values $\QQ(\sqrt{5})$ and cyclotomic deficiency $2$.) Our
work in \cite{Hung-Tiep23} suggests the following.

\begin{conjecture}[the cyclotomic deficiency conjecture]\label{conj:Hung-Tiep}
Let $\chi$ be an irreducible character of a finite group $G$. Then
\[[\QQ_{c(\chi)}:\QQ(\chi)]\leq \chi(1).\]
\end{conjecture}

\noindent Beyond Cram's result, further evidence in support of the
conjecture is known, including the cases of alternating and
symmetric groups, general linear and unitary groups, and irreducible
characters of prime degree; see \cite{HTZ}.

It is worth noting that the element-wise version of
Conjecture~\ref{conj:Hung-Tiep} -- namely, whether
$[\QQ_{c(\chi(g))}:\QQ(\chi(g))]\leq \chi(1)$ for $\chi\in\Irr(G)$
and $g\in G$ -- is true for degree-2 characters but fails in
general. This is a recent work by C.~Herbig \cite{He25}.

We find it interesting that Feit's conjecture and the cyclotomic
deficiency conjecture can be combined into a single statement, as
follows.

\begin{conjecture}\label{conj:stronger}
Let $G$ be a finite group and $\chi\in\Irr(G)$ be non-linear. Then
there exist a proper subgroup $H<G$ and $\psi\in\Irr(H)$ such that
\[
\QQ(\chi) \subseteq \QQ(\psi) \quad \text{and} \quad [\QQ(\psi) :
\QQ(\chi)] \le \frac{\chi(1)}{\psi(1)}.
\]
\end{conjecture}

\begin{proposition}\label{prop:1}
Conjecture~\ref{conj:stronger} is equivalent to the combination of
the cyclotomic deficiency conjecture and Feit's conjecture.
\end{proposition}

\begin{proof}
We first show that Conjecture~\ref{conj:stronger} implies the
cyclotomic deficiency conjecture. We proceed by induction on $|G|$.
Note that the cyclotomic deficiency conjecture is trivial for linear
characters, so we may assume that $\chi(1) > 1$. By hypothesis,
there exists $H < G$ and $\psi \in \Irr(H)$ such that $\QQ(\chi)
\subseteq \QQ(\psi)$ and $[\QQ(\psi) : \QQ(\chi)] \le
\chi(1)/\psi(1)$. It follows that
\[
[\QQ_{c(\chi)} : \QQ(\chi)] \le [\QQ_{c(\psi)} : \QQ(\chi)] =
[\QQ_{c(\psi)} : \QQ(\psi)] \, [\QQ(\psi) : \QQ(\chi)].
\]
By the induction hypothesis, we have $[\QQ_{c(\psi)} : \QQ(\psi)]
\le \psi(1)$. Hence,
\[
[\QQ_{c(\chi)} : \QQ(\chi)] \le \psi(1) \cdot
\frac{\chi(1)}{\psi(1)} = \chi(1),
\]
as desired.

To see that Conjecture~\ref{conj:stronger} also implies Feit's
conjecture, we apply Conjecture~\ref{conj:stronger} repeatedly to
obtain a linear character $\lambda$ of some subgroup $K$ of $G$ such
that
\[
\mathbb{Q}(\chi) \subseteq \mathbb{Q}(\lambda) =
\mathbb{Q}_{o(\lambda)}.
\]
Hence $c(\chi)$ divides $o(\lambda)$. Since $o(\lambda)$ equals the
order of some element of $K/K'$, it follows that $c(\chi)$ divides
the order of some element $x \in K$. Moreover, some power of $x$
will have order $c(\chi)$.

For the converse implication, let $g \in G$ be an element of order
$|g| = c(\chi) =: a$. If $G = \langle g \rangle$, then there is
nothing to prove. Otherwise, we may take $H = \langle g \rangle$ and
let $\psi$ be an irreducible character of $H$ of order $a$.
\end{proof}

Note that Conjecture~\ref{conj:stronger} is straightforward when
$\chi$ is \emph{imprimitive}; that is, when there exists a subgroup
$H < G$ and $\psi \in \Irr(H)$ such that $\chi = \psi^G$. For every
$\tau \in \Gal(\QQ(\psi)/\QQ(\chi))$, the restriction $\chi_H$ is
$\tau$-invariant, and hence $[\chi_H, \psi^\tau] = [\chi_H, \psi]$.
In other words, $\psi^\tau$ is also an irreducible constituent of
$\chi_H$ for each $\tau \in \Gal(\QQ(\psi)/\QQ(\chi))$. Since there
are exactly $[\QQ(\psi) : \QQ(\chi)]$ distinct
$\Gal(\QQ(\psi)/\QQ(\chi))$-conjugates of $\psi$, all of which have
the same degree as $\psi$, it follows that $\chi(1) \geq [\QQ(\psi)
: \QQ(\chi)] \, \psi(1)$, as desired.

We now turn to another problem concerning the cyclotomic deficiency.
In \cite[Theorem~2.7]{Navarro-Tiep21}, Navarro and Tiep found a
hidden role of the cyclotomic deficiency in the problem of
determining which abelian number fields can arise as fields of
values of characters of odd-order groups. Specifically, if $\FF$ has
odd conductor, then there exist a group $G$ of odd order and
$\chi\in\Irr(G)$ with $\QQ(\chi)=\FF$ if and only if the extension
$\QQ_{c(\FF)}/\FF$ has odd degree. Their proof relies on two key
facts: odd-order groups are solvable (by the Feit--Thompson
theorem~\cite{FT63}), and quasi-primitive irreducible characters of
solvable groups factor completely into products of so-called
$p$-special characters. (See \cite[Theorem~2.17]{I18} for details.
Recall that $\chi$ is \emph{$p$-special} if $\chi(1)$ is a power of
$p$ and the determinantal order of every irreducible constituent of
the restriction of $\chi$ to every subnormal subgroup of $G$ is a
$p$-power.)

The analogous statement for groups whose orders are not divisible by
a fixed prime $p>2$ appears substantially more difficult, as one can
no longer rely on solvability.

\begin{conjecture}[\cite{HNT25}]\label{conj:characters-p'-groups}
Let $p$ be a prime and $\FF/\QQ$ be an abelian field extension. Then
$\FF$ is the field of values of an irreducible character of a
$p'$-order group if and only if both $c(\FF)$ and
$[\QQ_{c(\FF)}:\FF]$ are not divisible by $p$.
\end{conjecture}

Already in 1972, B.~Fein and B.~Gordon proved that for any given
abelian extension $\FF/\QQ$, there exists a (solvable) group $G$ of
order $c(\FF)[\QQ_{c(\FF)}:\FF]$ and $\chi\in\Irr(G)$ such that
$\QQ(\chi)=\FF$ (see \cite[Theorem 8]{FG72}). This handles the
``if'' implication. Therefore,
Conjecture~\ref{conj:characters-p'-groups} follows from the
following, somewhat unexpected, statement.

\begin{conjecture}[\cite{HNT25}]
Let $G$ be a finite group and $\chi\in\Irr(G)$. Then
\[[\QQ_{c(\chi)}:\QQ(\chi)] \text{ divides } |G|.\]
\end{conjecture}

\noindent Perhaps surprisingly, our work in \cite{HNT25} shows that
this divisibility follows from the inductive Feit condition
discussed in Section~\ref{sec:conductor}.

It is remarkable that we have not been able to find a counterexample
for the inequality
$\nu_p(c(\chi))+\nu_p([\QQ_{c(\chi)}:\QQ(\chi)])\le \nu_p(|G|)$. If
this turns out to be true in general, then
Conjecture~\ref{conj:characters-p'-groups} could be extended to all
finite groups.


\section{The $p$-rationality}\label{sec:p-rationality}

We again fix a prime $p$ but now consider a $p$-local version of the
conductor which measures how rational at $p$ the values of a
character $\chi$ are. More specifically, for $\chi\in\Char(G)$, let
\[
\lev(\chi):=\nu_p(c(\chi)),
\]
where $\nu_p$ is the usual $p$-adic valuation as defined before. We
call it the \emph{$p$-rationality level} of $\chi$. Note that $\chi$
is more irrational at $p$ if its level is higher. In particular,
$\chi$ is \emph{$p$-rational} if $\lev(\chi)=0$ and $\chi$ is
\emph{almost $p$-rational} if $\lev(\chi)\leq 1$. As we will see,
this notion aries naturally in some global-local conjectures in the
field that have been studied extensively in the last two decades.

One of the first observations highlighting the importance of
$p$-rationality is due to Isaacs and Navarro
\cite{Isaacs-Navarro01}. Motivated by Brauer's Problem~12
\cite{Brauer63}, which asks how much information about the structure
of Sylow subgroups of a group $G$ can be extracted from the
character table of $G$, they formulated the following statement --
originally a conjecture, but now a theorem.

\begin{theorem}\label{IN01}
Let $p$ be a prime, $G$ a finite group and $P\in\Syl_p(G)$. Let
$\alpha\in\ZZ^{+}$. Then $\exp(P/P')\leq p^\alpha$ if and only if
the $p$-rationality level of every $p'$-degree irreducible character
of $G$ is at most $\alpha$.
\end{theorem}

The ``if'' direction was proved by Navarro and Tiep in
\cite[Theorem~B]{Navarro-Tiep19}. For the opposite direction, they
obtained a clean reduction to almost quasisimple groups: it is true
if it holds for these groups. Using this reduction, G.~Malle
\cite{Malle19} completed the proof for $p=2$. Theorem~\ref{IN01} is
now fully resolved by recent work of L.~Ruhstorfer and
A.\,A.~Schaeffer Fry \cite{RS25}. We will come back to this
impressive work later.

Theorem \ref{IN01} is only a small piece of the celebrated
McKay-Navarro conjecture we are about to describe. The well-known
McKay conjecture \cite{McKay}, which is now a theorem thanks to the
work of M.~Cabanes and B.~Sp\"{a}th \cite{CS24}, asserts that the
number of irreducible $p'$-degree characters of $G$ is equal to that
of the normalizer $\bN_G(P)$ of some $P\in\Syl_p(G)$; that is,
\[
|\Irr_{p'}(G)|=|\Irr_{p'}(\bN_G(P))|.
\]
Noticing the similarity between the values of $p'$-degree characters
in $G$ and $\bN_G(P)$, Navarro \cite{Navarro04} proposed that there
should exist a bijection between the two sets that commutes with the
action of certain Galois automorphisms. More precisely, let
$\mathcal{H}_p$ be the subgroup of the Galois group
$\mathcal{G}:=\Gal(\QQ_{|G|}/\QQ)$ consisting of automorphisms that
send every root of unity $\zeta\in\QQ_{|G|}$ of order not divisible
by $p$ to $\zeta^{q}$, where $q$ is a certain fixed power of $p$.
The \emph{McKay-Navarro conjecture} asserts that
\begin{quote}
\emph{there exists an $\mathcal{H}_p$-equivariant bijection from
$\Irr_{p'}(G)$ to $\Irr_{p'}(\bN_G(P))$.}
\end{quote}
The conjecture has been resolved for $p=2$ by Ruhstorfer and
Schaeffer Fry \cite{RS25a}, but it is still open for odd primes.

Since $\mathcal{H}_p$ contains
$\Gal(\QQ_{|G|}/\QQ_{p^\alpha|G|_{p'}})$ for every nonnegative
integer $\alpha\le \nu_p(|G|)$, the McKay-Navarro conjecture implies
that the numbers of characters in $\Irr_{p'}(G)$ and
$\Irr_{p'}(\bN_G(P))$ at every $p$-rationality level $\alpha$ are
the same:
\[
|\{\chi\in\Irr_{p'}(G): \lev(\chi)=
\alpha\}|=|\{\psi\in\Irr_{p'}(\bN_G(P)): \lev(\psi)= \alpha\}|.
\]

There are particular elements of the Galois group that conveniently
determine the $p$-rationality of characters. For each
$\alpha\in\ZZ^{+}$, we denote by $\sigma_{\alpha}$ the automorphism
in $\Gal(\QQ^{ab}/\QQ)$ that fixes roots of unity of order not
divisible by $p$ and maps every $p$-power root of unity $\xi$ to
$\xi^{1+p^\alpha}$. Abusing notation, for a specific $G$, we also
use $\sigma_{\alpha}$ for the restriction of this automorphism to
$\QQ_{|G|}$. It is not difficult to see that, if
$\chi\in\Irr_{p'}(G)$ is not $p$-rational, then $\lev(\chi)$ is
precisely the smallest positive integer $\alpha$ such that $\chi$ is
$\sigma_{\alpha}$-invariant.

Ruhstorfer and Schaeffer Fry \cite{RS25} recently proved a weaker
version of the McKay-Navarro conjecture, known as the
\emph{Isaacs-Navarro Galois conjecture}. This was first proposed in
\cite{Isaacs-Navarro02}. Let $\mathcal{H}_0$ be the subgroup of
$\mathcal{H}_p$ consisting of all Galois automorphisms that act
trivially on $p'$-roots of unity and have $p$-power order.

\begin{theorem}[Ruhstorfer-Schaeffer Fry]\label{thm:RSF}
Let $G$ be a finite group, $p$ be a prime dividing $|G|$, and
$P\in\Syl_p(G)$. Then there exists an $\mathcal{H}_0$-equivariant
bijection from $\Irr_{p'}(G)$ to $\Irr_{p'}(\bN_G(P))$.
\end{theorem}

Theorem~\ref{thm:RSF} yields several consequences that had
previously been conjectured as part of the McKay-Navarro conjecture.
For example, since $\sigma_{\alpha} \in \mathcal{H}_0$ for every
$\alpha \in \mathbb{Z}^{+}$, one obtains the following result.
(Recall that a character $\chi$ is almost $p$-rational if
$\lev(\chi) \le 1$. Some initial work on almost $p$-rational
characters, including bounds on their number, appears in
\cite{Hung-Malle-Maroti21}.)

\begin{theorem}[Ruhstorfer-Schaeffer Fry]\label{thm:RSF2}
Let $G$ be a finite group, let $p$ be a prime dividing $|G|$, and
let $P\in\Syl_p(G)$. Then for each $\alpha\geq 2$, the numbers of
characters of $p$-rationality level $\alpha$ in $\Irr_{p'}(G)$ and
$\Irr_{p'}(\bN_G(P))$ coincide. Moreover, the numbers of almost
$p$-rational characters in $\Irr_{p'}(G)$ and $\Irr_{p'}(\bN_G(P))$
are also equal.
\end{theorem}

Unfortunately, Theorem \ref{thm:RSF} is not enough to distinguish
the characters of levels $0$ and $1$. As far as we know, the
following consequence of the McKay-Navarro conjecture (as explained
above), remains open for odd primes. The case $p=2$ is already
known, thanks to \cite{RS25a}.

\begin{conjecture}
Let $G$ be a finite group, let $p$ be a prime dividing $|G|$, and
let $P\in\Syl_p(G)$. Then the numbers of $p$-rational characters in
$\Irr_{p'}(G)$ and $\Irr_{p'}(\bN_G(P))$ are equal.
\end{conjecture}

It would be ideal to describe, even conjecturally, the
$p$-rationality levels of the irreducible characters of a finite
group in terms of the structure of the group, but this appears to be
a highly nontrivial problem. For irreducible characters of degree
prime to $p$, the answer is hidden in the McKay-Navarro conjecture.
Theorem~\ref{IN01} shows that the \emph{maximal} $p$-rationality
level among the $p'$-degree irreducible characters of $G$ is
$\nu_p(\exp(P/P'))$, with the exceptional case in which the maximal
level is $0$ although $\nu_p(\exp(P/P')) = 1$.

What about the intermediate levels? In \cite{Hung22}, we observed,
again as a consequence of the McKay-Navarro conjecture, that the
$p$-rationality levels of $p'$-degree characters satisfy a
\emph{continuity property}, and proved it for $p=2$. This property
is now fully confirmed for all $p$, thanks to Ruhstorfer-Schaeffer
Fry's Theorem~\ref{thm:RSF2}.

\begin{theorem}\label{thm:continuity}
Let $p$ be a prime, $G$ a finite group and $\alpha\in \ZZ_{\geq 2}$.
If $G$ has an irreducible $p'$-degree character of $p$-rationality
level $\alpha$, then $G$ has irreducible $p'$-degree characters of
every level from $2$ to $\alpha$.
\end{theorem}

The following simply-stated result follows from Theorems~\ref{IN01}
and \ref{thm:continuity}, as shown in \cite[Theorem~2.5]{Hung22}.

\begin{theorem}
Let $p$ be a prime, $G$ a finite group and $\alpha\in \ZZ_{\geq 2}$.
Let $M$ be a subgroup of $G$ of $p'$-index. Then $G$ has an
irreducible $p'$-degree character of $p$-rationality level $\alpha$
if and only if $M$ does.
\end{theorem}

The McKay-Navarro conjecture has a block-wise version called the
Alperin-McKay-Navarro conjecture \cite{Navarro04}. In the same way
that the continuity property of the $p$-rationality level of
$p'$-degree irreducible characters of $G$ follows from the
McKay-Navarro conjecture, the following follows from the
Alperin-McKay-Navarro conjecture. This was first observed in
\cite{Hung22}.

\begin{conjecture}[\cite{Hung22}]
Let G be a finite group and $p$ be a prime. If the principal
$p$-block of $G$ contains an irreducible character of degree coprime
to $p$ with $p$-rationality level $\alpha\geq 2$, then it contains
irreducible $p'$-degree characters of every level from $2$ to
$\alpha$.
\end{conjecture}

This has been verified for $p=2$ by recent work of G.~Malle,
J.\,M.~Mart\'{\i}nez and C.~Vallejo \cite{MMV24}. They also reduced
it to a slightly stronger problem for simple groups, but the case of
odd primes remains open.

The continuity property fails at level $1$ in general. There are
many examples showing that all the $p'$-degree almost $p$-rational
characters of a finite group may become $p$-rational. Conjecture~B
of \cite{MMV24} offers an explanation for this phenomenon. Here, the
Frattini subgroup $\Phi(X)$ of a finite group $X$ is the
intersection of all maximal subgroups of $X$.

\begin{conjecture}[Malle-Mart\'{\i}nez-Vallejo]\label{conj:MMV}
Let $G$ be a finite group, $p$ a prime, and $P\in\Syl_p(G)$. Write
$K: = \bN_G(P)/\Phi(P)$ and $Q: = P/\Phi(P)$. Then the following are
equivalent:
\begin{enumerate}[\rm(i)]
\item No irreducible character of
G of degree coprime to $p$ has $p$-rationality level $1$.

\item For every $1\neq y \in Q$ we have:
\begin{itemize}
\item[(a)] $\bN_K(\langle y\rangle )/\bC_K(\langle y\rangle)\cong \Aut(\langle y\rangle)$, and

\item[(b)] $\bN_K(\langle y\rangle)/Q$ acts trivially on the set of conjugacy classes of
$\bC_K(\langle y\rangle)/Q$.
\end{itemize}
\end{enumerate}
\end{conjecture}

As proved in \cite[\S5]{MMV24}, Conjecture~\ref{conj:MMV} is also a
consequence of the McKay-Navarro conjecture. In particular, it holds
for $p$-solvable groups, as well as sporadic groups, symmetric and
alternating groups, and simple groups of Lie type in defining
characteristic \cite{BN21,Navarro04,R21}.


\section{$p$-Parts of character degrees}

As seen in Section~\ref{sec:cyclotomic-deficiency}, there are
connections between the conductor and the degree of an irreducible
character. We now fix a prime $p$ and examine the (local)
relationship between the $p$-parts of the conductor and the degree.

In this section and the next, if $K/\QQ$ and $L/\QQ$ are finite
abelian extensions, we write $KL$ to denote the smallest subfield of
$\overline{\QQ}$ that contains both $K$ and $L$. (This field is
often called the \emph{compositum} of $K$ and $L$.)

In \cite{Navarro-Tiep21}, Navarro and Tiep put forward the
following, which, if true, would classify all possible abelian
number fields that can be realized as the field of values of a
$p'$-degree irreducible character.

\begin{conjecture}[Navarro-Tiep]\label{conj:NT21}
Let $\FF$ be an abelian extension of $\QQ$. Then $\FF$ is the field
of values of a $p'$-degree irreducible character of some finite
group if and only if
\[ [\FF\QQ_{c(\FF)_p} : \FF] \text{ is not divisible by } p.\]
\end{conjecture}

In the same paper, the authors proved the ``if'' direction for all
$p$, and both directions for $p=2$. Moreover, the other direction
was reduced to quasisimple groups. The result for $p=2$ takes a
nicer form: for an odd-degree irreducible character $\chi$ with
conductor $c(\chi)=2^am$, where $m$ is odd, $\QQ_{2^a}\subseteq
\QQ(\chi)$. This result vastly improves on a previous result of
Isaacs et. al. \cite{ILNT}, where it was shown that if $\chi$ is an
odd-degree irreducible character that is not $2$-rational, then
$\QQ(\chi)$ contains the imaginary unit $i$.

We remark that Conjecture~\ref{conj:NT21} is a consequence of the
McKay-Navarro conjecture for $G$ and $p$ when $|G|_{p'}$ has no
prime divisor congruent to $1$ modulo $p$, but no other connections
are currently known. See \cite[Theorem~1.2]{G-Y23} for further
details.

The ``only if'' implication of Conjecture~\ref{conj:NT21} would
follow from the following deep conjecture.

\begin{conjecture}[Navarro-Tiep]\label{conj:NT21-more}
Let $p$ be a prime, $G$ a finite group, and $P\in\Syl_p(G)$. Let
$\chi$ be a $p'$-degree irreducible character of $G$. Then
\[\QQ_{c(\chi)_p}=\QQ_p(\chi_P).\]
\end{conjecture}

Of course this has no content if $\lev(\chi)\leq 1$, so we assume
that $\lev(\chi)\geq 2$. In that case, the equality
$\QQ_{c(\chi)_p}=\QQ_p(\chi_P)$ is equivalent to
$\lev(\chi)=\lev(\chi_P)$ for odd $p$. When $p=2$, one needs to show
furthermore that $i\in\QQ(\chi_P)$ on top of
$\lev(\chi)=\lev(\chi_P)$ in order to achieve the equality.

Conjecture~\ref{conj:NT21-more} has been proved recently by Isaacs
and Navarro \cite{Isaacs-Navarro24} for $p$-solvable groups. One of
the key ideas is to work with the following $p$-local invariant of
characters. For $\Psi\in\Char(G)$ and a nonnegative integer $i$, let
\[
\Delta_i(\Psi):=\sum_{\chi\in\Irr_i(G)} [\chi,\Psi] \chi,
\]
where $\Irr_i(G):=\{\chi\in\Irr(G): \lev(\chi)=i\}$, and define
\[
\ell(\Psi):=\max\{i\in\ZZ_{\geq 0}: \Delta_i(\Psi)(1) \notequiv 0
\bmod p\}
\]
if one of $\Delta_i(\Psi)$ has degree not divisible by $p$. It turns
out that, at least for $p$-solvable groups, if $\chi$ is a
$p'$-degree irreducible character of $G$ of $p$-rationality level at
least 2 and $P\in\Syl_p(G)$, then
\[
\lev(\chi)=\lev(\chi_P)=\ell(\chi_P).
\]

It is not clear to us whether this remains true for all finite
groups.

\begin{question}\label{conj:ell}
Let $p$ be a prime, $G$ a finite group, and $P\in\Syl_p(G)$. Let
$\chi\in\Irr_{p'}(G)$ with $\lev(\chi)\geq 2$. Is it true that
\[\lev(\chi)=\ell(\chi_P)?\]
\end{question}

\begin{theorem}\label{thm:1implies3}
An affirmative answer to Question~\ref{conj:ell} implies Conjecture~
\ref{conj:NT21-more}.
\end{theorem}

\begin{proof}
Recall that Conjecture \ref{conj:NT21-more} is trivial when
$\lev(\chi)\leq 1$. Let $\chi\in\Irr_{p'}(G)$ with
$\alpha:=\lev(\chi)=\ell(\chi_P)\geq 2$. By \cite[Lemma~4.2]{HS25},
it follows that $\lev(\chi_P)\geq 1$ and
\[\ell(\chi_P)= \lev(\chi)\geq \lev(\chi_P)\geq \ell(\chi_P),\]
and thus
\[\ell(\chi_P)=\lev(\chi_P)=\lev(\chi)=\alpha.\]

Now we just follow the arguments in the proof of
\cite[Theorem~4.4]{HS25}. By \cite[Lemma~7.1]{Navarro-Tiep21} we
have $\mathbb{Q}_p(\chi_P) \subseteq \mathbb{Q}_{p^\alpha}$. As
$[\mathbb{Q}_{p^\alpha} : \mathbb{Q}_p] = p^{\alpha-1}$, every
automorphism $\tau \in
\Gal(\mathbb{Q}_{p^\alpha}/\mathbb{Q}_p(\chi_P))$ has $p$-power
order. Moreover, $\tau$ fixes $\chi_P$, so it permutes the linear
constituents of $\chi_P$ of level $\alpha$.

Since $\ell(\chi_P) = \alpha$, the degree of $\Delta_\alpha(\chi_P)$
is not divisible by $p$, which means that the number of linear
constituents of $\chi_P$ of $p$-rationality level $\alpha$ is not
divisible by $p$. We deduce that at least one of these constituents
must be $\tau$-fixed. This forces $\tau$ to act trivially on
$\mathbb{Q}_{p^\alpha}$, as wanted.
\end{proof}

We record one consequence of Conjecture~\ref{conj:NT21-more} that is
still open.

\begin{conjecture}\label{conj:NT21-more-more}
Let $p$ be a prime, $G$ a finite group, and $P\in\Syl_p(G)$. Let
$\chi\in\Irr_{p'}(G)$. Then $\chi$ is almost $p$-rational if and
only if $\chi_P$ is almost $p$-rational. In particular, $\chi$ is
$2$-rational if and only if $\chi_P$ is $2$-rational.
\end{conjecture}

We have seen the interplay between the $p$-rationality and $p$-Sylow
restriction. The following question, now on induced characters, is
another problem that we do not know the solution for now.

\begin{question}\label{prob:lev-induced}
Let $p$ be a prime, $G$ a finite group, and $M\leq G$ such that
$|G:M|$ is not divisible by $p$. Let $\psi$ be a $p'$-degree
irreducible character of $M$ with $\lev(\psi)\geq 2$. Is it true
that
\[\lev(\psi^G)=\lev(\psi)?\]
\end{question}

This is likely to be true. The following result provides a small
piece of supporting evidence.

\begin{theorem}
Question \ref{prob:lev-induced} has an affirmative answer when $M$
is $p$-solvable.
\end{theorem}

\begin{proof}
We have $\lev(\psi^G)\geq\lev((\psi^G)_P)\geq \ell((\psi^G)_P)$, by
\cite[Lemma~4.2]{HS25}. Moreover, $\ell((\psi^G)_P)=\ell(\psi_P)$,
by following the proof of \cite[Lemma~4.3]{HS25}. By the work in
\cite{Isaacs-Navarro24} and the $p$-solvability of $M$, as mentioned
above, we also have $\lev(\psi)=\ell(\psi_P)$. Altogether, we deduce
that $\lev(\psi^G)\geq \lev(\psi)$. As $\QQ(\psi^G)\subseteq
\QQ(\psi)$ by the induced-character formula, the result follows.
\end{proof}

We already mentioned a result from \cite{ILNT} stating that if
$\chi$ is an odd-degree irreducible character, then either $\chi$ is
$2$-rational or the field $\mathbb{Q}(\chi)$ contains the imaginary
unit $i$. It follows that, if such a character has quadratic field
of values $\mathbb{Q}(\sqrt{d})$ for some square-free integer $d \ne
-1$, then necessarily $d \equiv 1 \pmod{4}$.

Further investigation into other quadratic number fields appears to
reveal a deeper connection: the quadratic irrationality of an
irreducible character $\chi$ reflects not only the parity of
$\chi(1)$ but also the higher $2$-divisibility of its degree. We
observe that $\mathbb{Q}(\sqrt{\pm 2})$ and $\mathbb{Q}(\sqrt{d})$
with $d \equiv 3 \pmod{4}$ can occur as fields of values of
irreducible characters $\chi$ whose degrees are exactly divisible by
$2$ (that is, divisible by $2$ but not by $4$). However, an
irreducible character with field of values $\mathbb{Q}(\sqrt{d})$
for $d \ne \pm 2$ even appears always to have degree divisible by
$4$.

The following, which generalizes
Conjecture~\ref{conj:Navarro-Tiep21-more}, provides an initial clue
of this phenomenon.

\begin{question}[Isaacs-Navarro]\label{conj:Isaacs-Navarro24}
Let $p$ be a prime, $G$ a finite group, and $P\in\Syl_p(G)$. Let
$\chi\in\Irr(G)$. Is it true that
\[[\QQ_{c(\chi)_p}:\QQ_p(\chi_P)]\leq \chi(1)_p?\]
\end{question}

This far-reaching problem, appeared as Question~D in
\cite{Isaacs-Navarro24}, implies the ``only if'' direction of the
next conjecture, which, if true, would classify the abelian number
fields that can arise as fields of values of irreducible characters
whose degrees have a prescribed $p$-part.

\begin{conjecture}[\cite{HNT25}]\label{conj:HNT25-a} Let $\FF$ be an abelian extension of $\QQ$ and $a:=\nu_p(c(\FF))$.
Let $b\in\ZZ_{\geq 0}$. Then $\FF$ is the field of values of an
irreducible character $\chi$ of some finite group with
$\nu_p(\chi(1))=b$ if and only if
\[ \nu_p([\FF\QQ_{p^a} : \FF])\leq b.\]
\end{conjecture}

The ``if'' implication was resolved in \cite{HNT25}. Since
$[\QQ_{p^a}(\chi): \QQ_p(\chi)]$ is exactly the $p$-part of
$[\QQ_{p^a}(\chi): \QQ(\chi)]$, the ``only if'' direction is in fact
equivalent to the following.

\begin{conjecture}[\cite{HNT25}]\label{conj:HNT25} Let $\chi$ be an irreducible character
of some finite group. Assume that $c(\chi)$ is divisible by $p$.
Then
\[\chi(1) \text{ is divisible by } [\QQ_{c(\chi)_p}(\chi): \QQ_p(\chi)].\]
\end{conjecture}

We note that Conjecture~\ref{conj:Isaacs-Navarro24} remains open
even for $p$-solvable groups. On the other hand,
Conjecture~\ref{conj:HNT25}, although known for such groups, has not
yet been reduced to the simple groups.

The following, appeared as Condition~D in \cite{Navarro-Tiep21},
strengthens the McKay-Navarro conjecture by incorporating the fields
of values of Sylow restrictions as in
Conjecture~\ref{conj:NT21-more}.

\begin{conjecture}[Navarro-Tiep]\label{conj:Navarro-Tiep21-more}
Let $p$ be a prime, $G$ a finite group, and $P\in\Syl_p(G)$. Then
there exists an $\mathcal{H}_p$-equivariant bijection \[^*:
\Irr_{p'}(G) \rightarrow \Irr_{p'}(\bN_G(P))\] such that
\[\QQ_p(\chi_P)=\QQ_p(\chi^\ast_P)\] for all $\chi\in\Irr_{p'}(G)$.
\end{conjecture}

We conclude this section with a note that the problems considered
here, as well as in Section~\ref{sec:cyclotomic-deficiency}, are
parts of a more general problem:
\begin{quote}
\emph{Determine the pairs $(d,\FF)$, where $d$ is a positive integer and
$\FF$ is an abelian number field, for which there exists a finite
group $G$ and $\chi\in\Irr(G)$ satisfying
\[
\begin{cases}
\QQ(\chi)=\FF,\\
\chi(1)=d.
\end{cases}
\]}
\end{quote}

For example, when $\FF$ is a cyclotomic field, the system has
solutions for every $d$ (see \cite[Lemma~2.1]{Hung-Tiep23}). Another
case is $d\in\{2,5\}$, where solutions exist precisely when
$[\QQ_{c(\FF)}:\FF]\in\{1,d\}$. For $d=3$, solutions occur exactly
when $[\QQ_{c(\FF)}:\FF]\in\{1,3\}$ or when $\FF=\QQ_k(\sqrt{5})$
for some $k\in\ZZ^+$ not divisible by $5$. (See
\cite[Theorem~1.1]{Hung-Tiep23} and the discussion on p.~18 of
\cite{HTZ}.) More generally, when $d$ is a prime, we show in
\cite[Theorem~D]{HTZ} that the system has solutions if and only if
$\FF=\EE \mathbb{K}$, where $\mathbb{K}$ is a cyclotomic field and
$\EE$ satisfies one of the following.
\begin{enumerate}
\item $\EE=\QQ$,

\item $[\QQ_{c(\EE)}:\EE]=d$, or

\item $\EE=\QQ(\chi)$, where $\chi$ is an irreducible character of degree $d$ of a
quasisimple group.
\end{enumerate}

If we look only at solvable groups, the answer is simple: solutions
exist if and only if $[\QQ_{c(\FF)}:\FF]$ divides $d$
(\cite[Theorem~2.2]{Hung-Tiep23}). For general groups,
Conjectures~\ref{conj:Hung-Tiep}, \ref{conj:NT21}, and
\ref{conj:HNT25-a} provide some necessary conditions for the system
to admit solutions. While a complete \emph{if and only if}
characterization appears out of reach at present, we are certain
that several additional interesting necessary conditions remain to
be discovered.


\section{Height zero characters}

In this final section, we turn to \emph{modular representation
theory} and discuss the role of the conductor in studying the values
of height-zero characters. Again, we fix a prime $p$ and a finite
group $G$.

Let $B$ be a $p$-block of $G$, and let $\Irr(B)$ denote the set of
ordinary irreducible characters in $B$. For $\chi \in \Irr(B)$, the
\emph{$p$-height} of $\chi$ is defined by
\[
\height(\chi) := \nu_p(\chi(1)) - \min_{\psi \in \Irr(B)}
\{\nu_p(\psi(1))\}.
\]
A character $\chi$ is said to have \emph{height zero} if
$\height(\chi) = 0$; in other words, its degree has the minimal
possible $p$-part among all irreducible characters in its block.

Let $D$ be a defect group of $B$. It is a certain $p$-subgroup of
$G$ of order $p^{d(B)}$, where $d(B)$ is the \emph{defect} of $B$,
defined as $d(B) := \nu_p(|G|) - \min_{\psi \in \Irr(B)}
\{\nu_p(\psi(1))\}$. Note that characters of $p'$-degree have height
zero in blocks of maximal defect, namely the defect equal to
$\nu_p(|G|)$. It is therefore reasonable to expect that some results
and conjectures on $p'$-degree characters in the previous sections
can be generalized to height zero characters.

The possible fields of values of $2$-height zero characters have
been classified by Navarro, Ruhstorfer, Tiep, and Vallejo
(\cite[Theorem~B]{NRTV24}): An abelian number field $\FF$ is the
field of values of a $2$-height zero irreducible character of some
finite group if and only if $\QQ_{c(\FF)}=\FF\QQ_{c(\FF)_{2'}}$. It
is still an open problem for general $p$.

\begin{conjecture}[Navarro-Ruhstorfer-Tiep-Vallejo]
Let $p$ be a prime and $\FF$ be an abelian extension of $\QQ$. Then
$\FF$ is the field of values of a $p$-height zero irreducible
character of some finite group if and only if
\[ [\QQ_{c(\FF)} : \FF\QQ_{c(\FF)_{p'}}] \text{ is not divisible by } p.\]
\end{conjecture}

It was already shown in \cite[Theorem~5.1]{NRTV24} that if
$[\QQ_{c(\FF)} : \FF\QQ_{c(\FF)_{p'}}]$ is not divisible by $p$,
then $\FF$ is the field of values of a $p$-height zero character of
some finite group. In the same paper, the authors proved that the
other direction in fact follows from the statement of the
aforementioned Alperin-McKay-Navarro conjecture and reduced it to a
problem on quasisimple groups.

We have seen from the previous section that the $p$-rationality of a
$p'$-degree character can be determined from the values of the
character at $p$-elements. It is possible to capture the
$p$-rationality of a height zero character inside a local subgroup?
The following offers an answer.

\begin{conjecture}[\cite{HS25}]\label{conj:Hung-SF}
Let $p$ be a prime, $G$ a finite group, and $\chi$ be a height zero
character in a block $B$ of $G$ with $\lev(\chi)\geq 2$. Suppose
that $D$ is a defect group of $B$. Then
\[\lev(\chi)=\lev(\chi_{\bN_G(D)}).\]
\end{conjecture}

This was proposed in \cite{HS25}, where several supporting cases
were presented. In particular, it was shown to hold for characters
in blocks with cyclic defect. In fact, the result in the
cyclic-defect case is notably stronger. The potential discrepancy
$\lev(\chi)=1$ and $\lev(\chi_{\mathbf{N}_G(D)})=0$ is possible in
general but cannot occur when $D$ is cyclic. Second, again when $D$
is cyclic, if an element of $G$ captures the $p$-rationality of
$\chi$, it must lie insider the normalizer $\bN_G(D)$ up to
$G$-conjugation (see \cite[\S3]{HS25}). We do not know if this is
true for arbitrary defect.

Conjecture~\ref{conj:Hung-SF} has several interesting consequences
(see \cite[\S7]{HS25}). We mention one of them here. The following
should be compared with Conjecture~\ref{conj:NT21-more-more}.

\begin{conjecture}[\cite{HS25}]
Let $\chi$ be a height-zero character of a finite group $G$ and $D$
a defect group of the $p$-block of $G$ containing $\chi$. Then
$\chi$ is almost $p$-rational if and only if $\chi_{\bN_G(D)}$ is
almost $p$-rational. In particular, $\chi$ is $2$-rational if and
only if $\chi_{\bN_G(D)}$ is $2$-rational.
\end{conjecture}

Speaking of almost $p$-rational characters, we note here a related
question connected to Brauer's Problem 12 (see the discussion before
Theorem~\ref{IN01}).

\begin{question}[\cite{Hung-Malle-Maroti21}, Question~1.5]
Let $G$ be a finite group and $p$ a prime dividing $|G|$. Is it true
that Sylow $p$-subgroups of $G$ are cyclic if and only if the number
of $p'$-degree almost $p$-rational irreducible characters in the
principal block of $G$ is of the form $e+\frac{p-1}{e}$ for some
divisor $e\in\ZZ^+$ of $p-1$?
\end{question}

Can Conjecture~\ref{conj:Hung-SF} be generalized to characters of
positive height? Remarkably, the only counterexamples we have found
to the inequality $\lev(\chi)-\lev(\chi_{\bN_G(D)})\leq
\height(\chi)$ are those with $\lev(\chi)=1$,
$\lev(\chi_{\bN_G(D)})=0$, and $\height(\chi)=0$.

Finally, in the spirit of Conjecture~\ref{conj:Navarro-Tiep21-more},
one can combine the Alperin-McKay-Navarro conjecture and
Conjecture~\ref{conj:Hung-SF} to obtain the following. (Recall that
$\mathcal{H}_p$ is the subgroup of $\Gal(\QQ_{|G|}/\QQ)$ consisting
of the Galois automorphisms that send every root of unity
$\zeta\in\QQ_{|G|}$ of order not divisible by $p$ to $\zeta^{q}$,
where $q$ is a certain fixed power of $p$. Also, if $B$ is a
$p$-block of a finite group $G$, then $\Irr_0(B)$ denotes the subset
of $\Irr(B)$ of all the height zero characters of $B$.)

\begin{conjecture}[\cite{HS25}]
Let $p$ be a prime and $G$ a finite group. Let $B$ be a $p$-block of
$G$ with defect group $D$ and $b$, a block of $\bN_G(D)$, be its
Brauer correspondent. Let $\mathcal{H}_B$ be the subgroup of
$\mathcal{H}_p$ fixing $B$. Then there exists an
$\mathcal{H}_B$-equivariant bijection \[^*: \Irr_0(B) \rightarrow
\Irr_0(b)\] such that \[\lev(\chi_{\bN_G(D)})=\lev(\chi^\ast)\] for
every $\chi\in\Irr_0(B)$ of $p$-rationality level at least $2$.
\end{conjecture}


\end{document}